\documentclass[11pt,reqno]{amsart}
\usepackage{amsmath,amssymb,amsthm,pinlabel, graphicx}

\usepackage[colorlinks=true,
            linkcolor=blue,
            urlcolor=cyan,
            citecolor=magenta]{hyperref}

\usepackage{fullpage}

\usepackage{tikz}
\usetikzlibrary{matrix,arrows,decorations.pathmorphing}

\theoremstyle{plain}
\newtheorem{theorem}{Theorem}
\newtheorem{lemma}[theorem]{Lemma}
\newtheorem{proposition}[theorem]{Proposition}
\newtheorem{corollary}[theorem]{Corollary}

\newtheorem*{claim*}{Claim}
\newtheorem*{case*}{Case}
\newtheorem{fact}[theorem]{Fact}

\theoremstyle{definition}
\newtheorem{definition}[theorem]{Definition}

\newtheorem{remark}[theorem]{Remark}

\newtheorem{convention}[theorem]{Convention}
\newtheorem{convention-remark}[theorem]{Convention-Remark}

\numberwithin{equation}{section}
\numberwithin{theorem}{section}

\newcommand{\fakeenv}{} 

\newenvironment{restate}[2]  
{ 
 \renewcommand{\fakeenv}{#2} 
 \theoremstyle{plain} 
 \newtheorem*{\fakeenv}{#1~\ref{#2}} 
 \begin{\fakeenv}
}
{
 \end{\fakeenv}
}

\newcommand{\NN}{\mathbb{N}} 

\newcommand{\PP}{\mathbb{P}}
\newcommand{\ZZ}{\mathbb{Z}} 

\newcommand{\calA}{\mathcal{A}}

\newcommand{\calF}{\mathcal{F}}

\newcommand{\calH}{\mathcal{H}}

 \newcommand{\g}{\mathfrak{g}}

\newcommand{\fF}{\mathcal{FF}}

\newcommand{\param}%
	{{\mathchoice{\mkern1mu\mbox{\raise2.2pt\hbox{$\centerdot$}}\mkern1mu}%
	{\mkern1mu\mbox{\raise2.2pt\hbox{$\centerdot$}}\mkern1mu}%
	{\mkern1.5mu\centerdot\mkern1.5mu}{\mkern1.5mu\centerdot\mkern1.5mu}}}

\DeclareMathOperator{\Curr}{Curr}
\newcommand{\PCurr}{\PP {\rm Curr}}
\DeclareMathOperator{\Aut}{Aut}

\DeclareMathOperator{\IA}{IA}
\DeclareMathOperator{\CT}{CT}

\DeclareMathOperator{\Out}{Out}

\DeclareMathOperator{\FN}{F_N}
\DeclareMathOperator{\F}{F}
\DeclareMathOperator{\RN}{R_N}
\DeclareMathOperator{\NEG}{NEG}
\DeclareMathOperator{\EG}{EG}
\DeclareMathOperator{\INP}{INP}

\DeclareMathOperator{\supp}{supp}
\DeclareMathOperator{\Stab}{Stab}
\DeclareMathOperator{\PF}{PF}
\DeclareMathOperator{\pINP}{pINP}

\begin{document}



\title[Hyperbolic extensions of free groups from atoroidal ping-pong]{Hyperbolic extensions of free groups from atoroidal ping-pong}

\author[C.~Uyanik]{Caglar Uyanik}
\address{Department of Mathematics \\
Yale University\\
10 Hillhouse Ave. \\
New Haven, CT 06511, USA}
\email{\href{mailto:caglar.uyanik@yale.edu}{caglar.uyanik@yale.edu}}

\thanks{The author gratefully acknowledges support from NSF grants DMS 1405146 and DMS 1510034}
\begin{abstract}
We prove that all atoroidal automorphisms of $\Out(\FN)$ act on the space of projectivized geodesic currents with generalized north-south dynamics.  As an application, we produce new examples of non virtually cyclic, free and purely atoroidal subgroups of $\Out(\FN)$ such that the corresponding free group extension is hyperbolic. Moreover, these subgroups 
are not necessarily convex cocompact. \end{abstract}

\maketitle



\section{Introduction}

Let $\FN$ be a free group of rank $N\ge3$, and $\Out(\FN)$ be its outer automorphism group. Consider the following short exact sequence 
\[
1\longrightarrow\FN\stackrel{\iota}{\longrightarrow}\Aut(\FN)\stackrel{q}{\longrightarrow}\Out(\FN)\longrightarrow1
\]
where $\iota$ sends an element of $\FN$ to the corresponding inner automorphism, and $q$ is the natural quotient map.

Given a subgroup $\Gamma<\Out(\FN)$, the preimage $E_{\Gamma}=q^{-1}(\Gamma)$ gives an extension of $\FN$. In fact, any extension of $\FN$ produces an extension of this form, see \cite[Section 2]{DTHyp}.
Motivated by a long history of investigating hyperbolic extensions of hyperbolic groups \cite{BF, BFH97, FM, Ha, shadows, Mosher}, Dowdall and Taylor \cite{DTHyp} initiated a systematic study of the following question: 
\[
\textit{What conditions on the group $\Gamma$ guarantees that the extension group $E_{\Gamma}$ is hyperbolic?} 
\]

When the group $\Gamma$ is infinite cyclic, generated by $\varphi\in\Out(\FN)$, combined work of Bestvina-Feighn \cite{BF} and Brinkmann \cite{Brink} shows that $E_{\Gamma}$ is hyperbolic if and only if $\varphi$ is \emph{atoroidal}, meaning that no power of $\varphi$ fixes a non-trivial conjugacy class in $\FN$. Dowdall and Taylor \cite{DTHyp} proved that if a finitely generated subgroup $\Gamma<\Out(\FN)$ is purely atoroidal (i.e. every infinite order element is atoroidal), and the orbit map from $\Gamma$ into the free factor complex is a quasi-isometric embedding, then the extension $E_{\Gamma}$ is hyperbolic. The second condition also implies that every infinite order element  $\varphi\in\Gamma$ is \emph{fully irreducible}, meaning that no power of $\varphi$ fixes the conjugacy class of a proper free factor, see section \ref{prelim} for definitions.   

So far the only known examples of hyperbolic extensions of free groups come from slight variations, or iterated applications of these two constructions, and Shottky-type subgroups generated by high powers of fully irreducible and atoroidal elements. The following subgroup alternative result allows us to produce more examples: 

\begin{theorem}\label{maintwo} Let $\calH<\Out(\FN)$ be a subgroup that contains an atoroidal element $\varphi$. Then one of the following occurs:
\begin{enumerate}
\item
There is some minimal $\calH$-invariant free factor $A$ of $\FN$ such that the restriction of $\calH$ to $A$ is virtually cyclic in $\Out(A)$.  
\item There exists a subgroup $\Gamma\le \calH$ such that $\Gamma\cong \F_2$ and that every nontrivial element of $\Gamma$ is atoroidal. Moreover, the corresponding extension group $E_\Gamma$ is hyperbolic.
\end{enumerate}
\end{theorem}

\begin{remark} Note that Theorem \ref{maintwo} generalizes a well known result of Bestvina-Feighn-Handel. Indeed, if the subgroup $\calH$ is irreducible, namely no finite index subgroup of $\calH$ fixes a proper free factor, then $\calH$ contains a fully irreducible element by a theorem of Handel--Mosher \cite{HMIntro}, see also \cite{Hshort}. Since $\calH$ contains an atoroidal element, then \cite[Theorem 5.4]{UyaNSD} implies that $\calH$ contains an element which is both fully irreducible and atoroidal. In that case, Bestvina-Feighn-Handel \cite{BFH97} show that either $\calH$ is virtually cyclic, or there is a nonabelian free subgroup $\Gamma$ of $\calH$ such that every non-trivial element of $\Gamma$ is atoroidal, and the corresponding free group extension is hyperbolic. A different proof of the aforementioned result of Bestvina--Feighn--Handel is given by Kapovich and Lustig \cite {KL5} who additionally obtained that each nontrivial element is fully irreducible. 
\end{remark}

\begin{remark} We remark that the subgroup $\calH$ isn't necessarily irreducible or it doesn't have to preserve a free splitting of $\FN$. Theorem \ref{maintwo} gives new examples of hyperbolic extensions of free groups, which do not come from previously known constructions. In particular, they are not necessarily convex cocompact \cite{DTcosurface, DTHyp, TT}. 
\end{remark}

The main ingredient in the proof of Theorem \ref{maintwo} is the following dynamical result. See section \ref{n-sdynamics} for definitions. 
\begin{theorem}\label{dynamicsofhyp} Let $\varphi\in \Out(\FN)$ be an atoroidal outer automorphism of a free group of rank $N\ge3$. Then, there exist a simplex of attraction $\Delta_{+}$ and a simplex of repulsion $\Delta_{-}$ in $\PCurr(\FN)$ such that $\varphi$ acts on $\PCurr(\FN)$ with generalized north-south dynamics from $\Delta_{-}$ to $\Delta{+}$. 
\end{theorem}

The space $\PCurr(\FN)$ of \emph{projectivized geodesic currents} contains positive multiples of conjugacy classes as a dense subset, and hence serves as a natural tool for detecting atoroidal outer automorphisms, see section \ref{currents} for details. The proof of Theorem \ref{dynamicsofhyp} builds on our earlier results with M. Lustig about dynamics of reducible substitutions \cite{LU1} and is modeled on the proof of the specific case where both $\varphi$ and $\varphi^{-1}$ admit absolute train track representatives as treated in \cite{LU2}. In this paper, we use completely split relative train track maps ($\CT's$) which are particularly nice topological representatives introduced by Feighn and Handel \cite{FH}. The new key insight in the proof of Theorem \ref{dynamicsofhyp} is to use the legal structure coming from the splitting units in the $\CT$ that represents $\varphi\in\Out(\FN)$ rather than using the classical legal structure coming from the edges. 

\medskip

As a byproduct of Theorem \ref{dynamicsofhyp} we also obtain:

\begin{corollary}\label{irreduciblenotcoco} Let $\varphi\in\Out(\FN)$ be a fully irreducible and atoroidal outer automorphism. Then for any atoroidal outer automorphism $\psi\in\Out(\FN)$ (not necessarily fully irreducible) which is not commensurable with $\varphi$ (i.e. $\varphi^{t}\neq\psi^{s}$ for any $s,t$), there exists an exponent $M>0$ so that for all $n,m>M$,  the subgroup $\Gamma=\langle\varphi^{n},\psi^m\rangle<\Out(\FN)$ is purely atoroidal and the corresponding extension $E_{\Gamma}$ is hyperbolic. 
\end{corollary}

Note that the subgroup $\Gamma$ in Corollary \ref{irreduciblenotcoco} is irreducible, and since $\Gamma$ is not purely fully irreducible the orbit map to the free factor graph is not a quasi-isometric embedding \cite{DTHyp}. 
\medskip

\noindent\textbf{Acknowledgments:} 
The author is grateful to Martin Lustig for useful discussions that helped form the core of this paper. He is grateful to his advisors Chris Leininger and Ilya Kapovich for support, encouragement and feedback. He also thanks Mladen Bestvina and Spencer Dowdall for useful discussions. Finally, he thanks the anonymous referees for useful comments and corrections. 

\section{Preliminaries}\label{prelim}

\subsection{Graphs and graph maps}

 A \emph{graph} $G$ is a $1$-dimensional cell complex, where $0$-cells are called \emph{vertices} and $1$-cells are called topological edges. We denote the set of vertices by $VG$, and the set of topological edges by $E_{top}G$. Identifying the interior of an edge with the open interval $(0,1)$ each edge admits exactly two orientations. We denote the set of \emph{oriented} edges by $EG$. Choosing an orientation on each edge splits the set $EG$ into two disjoint sets:  the set $E^{+}G$ of positively oriented and the set $E^{-}G$ of negatively oriented edges. Given an oriented edge $e\in EG$, the initial vertex of $e$ is denoted by $o(e)$ and the terminal vertex of $e$ is denoted by $t(e)$, and the edge with the opposite orientation is denoted by $e^{-1}$. 

An  \emph{edge path} $\gamma$ in $G$ is a concatenation $\gamma=e_1e_2\ldots e_n$ of edges in $G$ such that $t(e_{i-1})=o(e_i)$ for all $i=2, \ldots n$. An edge path $\gamma=e_1e_2\ldots e_n$ is called \emph{reduced (or tight)} if $e_{i}^{-1}\neq e_{i+1}$ for all $i=1,\ldots n-1$. A reduced edge path $\gamma=e_1e_2\ldots e_n$ is called cyclically reduced if $o(\gamma)=t(\gamma)$ and in addition $e_{n}^{-1}\neq e_1$. We call cyclically reduced edge paths \emph{circuits}.  

Given an edge path $\gamma$ we denote the reduced edge path obtained by a homotopy relative to endpoints of $\gamma$ by $[\gamma]$. 

\subsection{Markings and topological representatives}

Let $\RN$ denote the rose with $N$ pedals, which is the finite graph with one vertex and $N$ loop edges attached to that vertex. A \emph{marking} is a homotopy equivalence $m:\RN\to G$ where $G$ is a finite graph all of whose vertices are at least valence $2$. 

A homotopy equivalence $f:G\to G$ is a \emph{(topological) graph map} if it sends vertices to vertices, and its restriction to the interior of an edge is an immersion. Let $m':G\to\RN$ be a homotopy inverse to the marking $m:\RN\to G$. We say that a topological graph map is a \emph{topological representative} of an outer automorphism $\varphi\in\Out(\FN)$ if the induced map satisfies $(m'\circ f \circ m)_{\ast}:\FN\to\FN=\varphi$.  

A \emph{filtration} for a topological representative $f:G\to G$ is an ascending sequence of $f$-invariant subgraphs $\emptyset=G_0\subset G_1\subset \ldots\subset G_k=G$. The closure of $G_r\setminus G_{r-1}$ is called the \emph{$r^{th}$-stratum}, and is denoted by $H_r$. 

For each stratum $H_r$, there is an associated \emph{transition matrix} $M_{r}$ of $H_{r}$ which is a non-negative integer square matrix. The ${ij}^{th}$ entry of $H_{r}$ records the number of times $[f(e_i)]$ crosses $e_j$ or $e_j^{-1}$. A non-negative square matrix $M$ is called \emph{irreducible} if for each $i,j$, there exists $k=k(i,j)$ such that $M^{k}_{ij}>0$, the matrix $M$ is called \emph{primitive} if $k$ can be chosen independent of $i,j$. The stratum $H_r$ is called \emph{irreducible (resp. primitive)} if and only if $M_r$ is irreducible (resp. primitive). If $M_r$ is irreducible then it has a unique eigenvalue $\lambda\ge1$ called the Perron-Frobenius $(\PF)$ eigenvalue , for  which  the associated eigenvector is positive.  We say that $H_r$ is an \emph{exponentially growing stratum} or $\EG$ stratum if $\lambda>1$ and \emph{non-exponentially growing stratum} or $\NEG$ stratum if $\lambda=1$. We say that $H_r$ is a \emph{zero stratum} if $M_r$ is the zero matrix.

\subsection{Train track maps}

We first set up the relevant terminology to define relative train track maps, and their strengthened versions $\CT$'s. The standard resources for this section are \cite{BH92, BFH00, FH}. 

Let $f:G\to G$ be a topological graph map. A \emph{direction} at a point $v\in G$ is the germ of an initial segment of an oriented edge. The map $f:G\to G$ induces a natural \emph{derivative map} $Df$ on the set of germs, and we say that a direction is \emph{fixed} or \emph{periodic} if it is fixed or periodic under the derivative map. A \emph{turn} in $G$ is an unordered pair of directions. We say that a turn is \emph{degenerate} if the two directions are the same, and \emph{non-degenerate} otherwise. A turn is called \emph{illegal} if its image under some iterate of $Df$ is degenerate, otherwise a turn is called \emph{legal}. An edge path $\gamma=e_1e_2\ldots e_k$ is called legal if each turn $(e_i^{-1},e_{i+1})$ is legal. We say that, a turn is contained in a stratum $H_r$ if both directions are contained in $H_r$. An edge path $\gamma$ is called \emph{$r$-legal}, if every turn in $\gamma$ that is contained in $H_r$ is legal. If $H_r$ is an $\EG$ stratum, and $\gamma$ is a path whose endpoints are in $H_r\cap G_{r-1}$, then $\gamma$ is called a \emph{connecting path}. 

\begin{definition} A homotopy equivalence $f:G\to G$ representing $\varphi\in\Out(\FN)$ is called a \emph{relative train track} map, if for every exponentially growing stratum $H_r$ the following hold:
\begin{enumerate}
\item[(RTT-i)]  $Df$ maps the set of directions in $H_r$ to itself. 
\item[(RTT-ii)] For each connecting path $\gamma$ for $H_r$, $[f(\gamma)]$ is a connecting path for $H_r$. In particular, $[f(\gamma)]$ is nontrivial. 
\item[(RTT-iii)] If $\gamma$ is $r$-legal, then $[f(\gamma)]$ is $r$-legal. 
\end{enumerate}
\end{definition}

\begin{definition}[Nielsen paths]A path $\rho$ is a \emph{periodic Nielsen path} if there is an exponent $k\ge1$ such that $[f^{k}(\rho)]=\rho$. The mimimal such $k$ is called the \emph{period}, and if $k=1$ then $\rho$ is called a \emph{Nielsen path}. A periodic Nielsen path is called \emph{indivisible} if it cannot be written as a concatenation of periodic Nielsen paths. We will denote the (periodic) indivisible Nielsen paths by ($\pINP$) $\INP$'s. 
\end{definition}

\begin{definition}[Taken and exceptional paths] A path $\gamma\in G$ is called \emph{r-taken} by $f:G\to G$ if $\gamma$ appears as a subpath of $f^{k}(e)$ for some $k\ge1$ and for some edge $e\in H_r$ in an irreducible stratum. We will drop $r$ and only say \emph{taken} whenever $r$ is irrelevant.  Let $e_i, e_j$ be linear edges as defined in Definition \ref{NEGedges} below such that $f(e_i)=e_iw^{m_i}$ and $f(e_j)=e_jw^{m_j}$ for some root free Nielsen path $w$. Then a path of the form $e_iw^{p}e_j^{-1}$ for $p\in\ZZ$ is called an \emph{exceptional path}. 
\end{definition}

\begin{definition}[Splittings and complete splittings] Let $f: G\to G$ be a relative train track map. A decomposition of a path $\gamma$ in $G$ into subpaths $\gamma=\gamma_1\cdot\gamma_2\cdot\ldots\cdot\gamma_m$ is called a \emph{splitting} if $[f^{k}(\gamma)]=[f^{k}(\gamma_1)][f^{k}(\gamma_2)]\ldots[f^{k}(\gamma_m)]$. Namely, one can tighten the image $f^{k}(\gamma)$ by tightening the images of the subpaths $\gamma_i$. We use the ``$\cdot$'' notation for splittings. 

A splitting $\gamma=\gamma_1\cdot\gamma_2\cdot\ldots\cdot\gamma_m$ is called a \emph{complete splitting} if each term $\gamma_i$ is one of the following:
\begin{enumerate}
\item an edge in an irreducible stratum
\item an $\INP$
\item an exceptional path
\item a connecting path in a zero stratum that is both maximal and taken
\end{enumerate} 

The paths in the above list are called \emph{splitting units}. 
\end{definition}

\begin{lemma}\cite{BFH00, FH} Every completely split path or circuit has a unique complete splitting.
\end{lemma}

The properties of relative train track maps are not strong enough for our purposes. Hence, in order to study the dynamics of atoroidal outer automorphisms, we utilize \emph{completely split train track maps} ($\CT$'s) introduced by Feighn and Handel. In what follows, rather than giving the defining properties of $\CT$'s we will list the relevant properties of $\CT$'s and cite the appropriate resources. We refer reader to \cite{FH} for a detailed discussion of $\CT$'s. We begin with two definitions:

\begin{definition}\label{freefactors} A subgroup $\F<\FN$ is called a \emph{free factor} of $\FN$ if there is another subgroup $\F'<\FN$ such that $\F\ast \F'=\FN$. We denote the conjugacy class of a free factor $\F$ with $[\F]$.  A \emph{free factor system} is a collection $\calF=\{[\F^{1}],\ldots, [\F^{k}]\}$ of conjugacy classes of free factors of $\FN$ such that there exists $\F'<\FN$ (possibly trivial) with the property that $\FN=\F^{1}\ast\ldots\ast \F^k\ast \F'$. There is a partial order on the set of free factor systems as follows: Given two free factor systsems $\calF=\{[\F^{1}],\ldots, [\F^{k}]\}$ and $\calF'=\{[\F'^{1}],\ldots, [\F'^{l}]\}$ we say that $\calF\sqsubset\calF'$ if for each $[\F^{i}]\in\calF$ there exists  $[\F'^{j}]\in\calF'$ such that $g\F^{i}g^{-1}<F'^{j}$ for some $g\in\FN$. 
\end{definition}

The \emph{free factor graph} $\fF(\FN)$ is the (infinite) graph whose vertices correspond to conjugacy classes of proper free factors, and there is an edge between $[\F]$ and $[\F']$ if either $\F<g\F'g^{-1}$ or $\F'<g\F g^{-1}$ for some $g\in\FN$. By declaring the length of each edge $1$, $\fF(\FN)$ is equipped with a path metric $d$, and a result of Bestvina and Feighn says that $\fF(\FN)$ is hyperbolic \cite{BF14}. The group $\Out(\FN)$ acts on $\fF(\FN)$ with simplicial isometries and fully irreducible elements are precisely the loxodromic isometries \cite{BF14}. 

\begin{definition}For any marked graph $G$ and a subgraph $K$ of $G$ the fundamental group of the non-contractible components of $K$ determines a free factor system $[\pi_1(K)]=\mathcal{F}$ of $\FN$. We say that $K$ \emph{realizes} $\mathcal{F}$. Given a nested sequence $\mathcal{C}$ of free factor systems $\calF^1\sqsubset\calF^2\sqsubset\ldots\sqsubset\calF^n$ we say that $\mathcal{C}$ is \emph{realized} by a relative train-track map $f:G\to G$ if there is an $f$-invariant filtration $\emptyset=G_0\subset G_1\subset \ldots\subset G_k=G$ such that or all $1\le i\le n$ we have $\calF^i=[\pi_1(G_{k(i)})]$ for some $k(i)$. 
\end{definition}

The following theorem is the main existence result about $\CT$'s:

\begin{theorem}\cite[Theorem 4.28, Lemma 4.42]{FH}\label{CTsExist} There exists a uniform constant $M=M(N)\ge1$ such that for any $\varphi^{M}\in \Out(\FN)$, and any nested sequence $\mathcal{C}$ of $\varphi^{M}$-invariant free factor systems, there exists a $CT$ $f:G\to G$ that represents $\varphi^{M}$ and realizes $\mathcal{C}$. 
\end{theorem}

We now state several results about structures of paths in $\CT$'s that will be relevant in the discussion follows. 

\begin{lemma}\cite[Lemma 4.21]{FH} If $f:G\to G$ is a $\CT$, then every $\NEG$ stratum $H_r$ consists of a single edge $e_i$. Moreover, either $e_i$ is fixed, or $f(e_i)=e_i\cdot u_i$ where $u_i$ is a nontrivial, completely split circuit in $G_{i-1}$. 
\end{lemma}

\begin{definition}\label{NEGedges} Let $e\in G$ be an $\NEG$ edge. The edge $e$ is called a \emph{fixed edge} if $f(e)=e$, a \emph{linear edge} if $f(e)=e\eta$ where $\eta$ is a non-trivial Nielsen path, and a \emph{superlinear edge} otherwise. 
\end{definition}

\begin{lemma}[Properties of $\CT$'s]\label{CTmaps}\label{INPFacts}\cite[Definition 4.7, Lemma 4.13, Lemma 4.15, Corollary 4.19, Lemma 4.25]{FH}.
\
\begin{enumerate} 
\item \label{splits1} For each edge $e$ in an irreducible stratum, $f(e)$ is completely split. For each taken connecting path $\gamma$ in a zero stratum, $[f(\gamma)]$ is completely split. 
\item For each filtration element $G_r$, $H_r$ is a zero stratum if and only if $H_r$ is a contractible component of $G_r$. In particular, there are only finitely many reduced connecting paths that are contained in some zero stratum. 
\item Every periodic indivisible Nielsen path $(\INP)$ has period one. 
\item The endpoints of all $\INP$'s are vertices. The terminal endpoint of each $\NEG$ edge is fixed. 
\item \label{splits2} If $\gamma$ is a circuit or an edge-path, then $[f^{k}(\gamma)]$ is completely split for all sufficiently large $k$.
\item \label{zerolink} Each zero stratum $H_i$ is enveloped by an $\EG$ stratum $H_r$, each edge in $H_i$ is $r$-taken, and each vertex in $H_i$ is contained in $H_r$ and has link contained in $H_i\cup H_r$. 
\item If $H_r$ is an $\EG$ stratum, then 
there is at most one indivisible Nielsen path $\rho_r$ of height $r$ that intersects $H_r$ nontrivially. The initial edges of $\rho_r$ and $\rho_r^{-1}$ are distinct edges in $H_r$.  
\item If $H_r$ is a zero stratum or an $\NEG$ superlinear stratum, then no Nielsen path crosses an edge of $H_r$.

\end{enumerate}

\end{lemma}

\subsection{CT's representing atoroidal automorphisms}

Given an atoroidal outer automorphism $\varphi\in \Out(F_N)$, let $f:G\to G$ be a $\CT$ with filtration $\emptyset=G_0\subset G_1\subset \ldots\subset G_k=G$ that represents a suitable power of $\varphi$ as given by Theorem \ref{CTsExist}.
Observe that for such a $\CT$, there are no exceptional paths in the complete splitting of $[f^{n}(e)]$ for any $e\in\Gamma$ as there are no linear edges in $\Gamma$ (since it requires a closed Nielsen path). The following is an easy consequence of definitions: 

\begin{fact} Let $f:G\to G$ be a $\CT$ that represents an atoroidal outer automorphism. Then, every Nielsen path is a legal concatenation of $\INP$'s and fixed edges. 
\end{fact}

\begin{definition} We call a splitting unit $\sigma$ \emph{expanding} if $|[f^{n}(\sigma)]|\to\infty$ as $n\to\infty$. If $f:G\to G$ is a $\CT$ that represents an atoroidal outer automorphism, then an expanding splitting unit is one of the following three types:
\begin{enumerate}
\item an edge in an $\EG$ stratum
\item a superlinear edge in an $\NEG$ stratum 
\item a maximal connecting path $\gamma$ in a zero stratum such that the complete splitting of $[f^{k}(\gamma)]$ contains at least one of the above two types for some $k\ge1$. 
\end{enumerate}
\end{definition}

\subsection{Geodesic currents}\label{currents}

Let $\partial\FN$ denote the Gromov boundary of $\FN$. Let $\partial^{2}\FN$ be the double boundary, i.e.  $\partial^{2}\FN=\partial\FN\times\partial\FN\setminus\Delta$ where $\Delta$ denotes the diagonal. Let $\iota:\partial^{2}\FN\to\partial^{2}\FN$ be the \emph{flip} map given by $\iota(\zeta_1,\zeta_2)=(\zeta_2,\zeta_1)$. 

The group $\FN$ acts on itself by left multiplication which induces an action of $\FN$ on $\partial\FN$ and hence on $\partial^{2}\FN$. A \emph{geodesic current} on $\FN$ is a locally finite (positive) Borel measure on  $\partial^{2}\FN$ which is both $\FN$-invariant and flip-invariant. 

The space of geodesic currents on $\FN$ is denoted by $\Curr(\FN)$, and endowed with the weak-* topology it is a metrizable topological space \cite{Bosurvey}. The space of \emph{projectivized geodesic currents} $\PCurr(\FN)$ is the quotient of $\Curr(\FN)$ where two currents are equivalent if they are positive scalar multiples of each other. The space $\PCurr(\FN)$ is compact, see \cite{Ka2}.

Both $\Aut(\FN)$ and $\Out(\FN)$ acts on the space of currents by homeomorphisms, and these actions descend to well defined actions on $\PCurr(\FN)$. 

Let $g\in\FN$ be an element which is not a proper power.  We define the \emph{counting current} $\eta_g$ 
corresponding to $g$ as follows: for any Borel set $S\subset\partial^{2}\FN$
the value 
$\eta_g(S)$ is the number of $\FN$-translates of $(g^{-\infty},g^{\infty})$ or of  $(g^{\infty},g^{-\infty})$ 
that 
are contained in 
$S$. For any nontrivial element
$h\in\FN$ 
we write $h=g^{k}$, where $g$ is not a proper power, and set 
$\eta_{h}:=k\eta_{g}$.
A \emph{rational current} is a non-negative real multiple of a counting current. The set of rational currents forms a dense subset of $\Curr(\FN)$, see \cite{Ka1,Ka2, Martin}.

\section{Dynamics of atoroidal automorphisms}\label{dynamicsofatoroidal}

\subsection{North-south dynamics}\label{n-sdynamics}
Let $X$ be a compact metric space, and $G$ be a group acting on $X$ by homeomorphisms. We say that $g\in G$ acts on $X$ with \emph{(uniform) north-south dynamics} if the action of $g$ on $X$ has two distinct fixed points $x_-$ and $x_+$ and for any open neighborhood $U_{\pm}$ of $x_{\pm}$ and a compact set $K_{\pm}\subset X\setminus x_{\mp}$ there exists $M>0$ such that 
\[
g^{\pm n}K\subset U_{\pm}
\] 
for all $n\ge M$. 

North-south dynamics is a strong form of stability property for the action of a group on a compact metric space, and allows one to deduce various structural results about the group itself. For example, a fully irreducible outer automorphism $\varphi\in \Out(\FN)$ acts on the the closure $\overline{CV}$ of the projectivized outer space with north-south dynamics \cite{LL}. Similarly, if $\varphi$ is both fully irreducible and atoroidal, then $\varphi$ acts on $\PCurr(\FN)$ with north-south dynamics, \cite{Martin}, see also \cite{Uyaiwip}. On the other hand, an atoroidal outer automorphism does not act on $\PCurr(\FN)$ with classical north-south dynamics. Existence of invariant free factors makes them dynamically more complicated but, as we show below, they still exhibit a strong form of stability.  

\begin{definition}[Generalized north-south dynamics] Let $X$ be a compact metric space, and $G$ be a group acting on $X$ by homeomorphisms. We say that an element $g\in G$ acts on $X$ with generalized north-south dynamics if the action of $g$ on $X$ has two invariant disjoint sets $\Delta_{-}$, and $\Delta_{+}$, (i.e. $g\Delta_{-}=\Delta_-$ and $g\Delta_+=\Delta_+$) and for any open neighborhood $U_{\pm}$ of $\Delta_{\pm}$ and a compact set $K_{\pm}\subset \PCurr(\FN)\setminus\Delta_{\mp}$, there exists $M>0$ such that 
\[
g^{\pm n}K_{\pm}\subset U_{\pm}
\]
for all $n\ge M$. 
\end{definition}

\noindent We restate Theorem \ref{dynamicsofhyp} from the introduction for the benefit of the reader, the proof of which is given at the end of this section. 

\begin{restate}{Theorem}{dynamicsofhyp}
 Let $\varphi\in \Out(\FN)$ be an atoroidal outer automorphism of a free group of rank $N\ge3$. Then, there exist a simplex of attraction $\Delta_{+}$ and a simplex of repulsion $\Delta_{-}$ in $\PCurr(\FN)$ such that $\varphi$ acts on $\PCurr(\FN)$ with generalized north-south dynamics from $\Delta_{-}$ to $\Delta{+}$. 
\end{restate}

The rest of this section is modeled on our earlier paper \cite{LU2} with M. Lustig, and utilizes the dynamics of reducible substitutions as treated in \cite{LU1}. In what follows we explain the subtleties that arise in this new setting carefully, while referring to \cite{LU2} for arguments that follow by straightforward modifications from the old setting.

\subsection{Symbolic dynamics and $\CT$'s} 

In this section we recall the relevant definitions in symbolic dynamics and results from our earlier paper \cite{LU1}, that allows us to describe the simplex of attraction and simplex of repulsion in Theorem \ref{dynamicsofhyp} explicitly. 

Let $A=\{a_1,\ldots, a_n\}$ be a finite alphabet, and $A^{*}$ denote the set of all finite words in $A$. A \emph{substitution} $\xi: A\to A^*$ is a rule that assigns each letter $a\in A$ a non-empty word $w$ in $A^{*}$. A substitution induces a map, which we also denote with $\xi$, on the set of infinite words $A^{\NN}$ by concatenation: 
\[
\xi:A^{\NN}\to A^{\NN}
\]
\[
x_1x_2\ldots \mapsto \xi(x_1)\xi(x_2)\ldots\]

Given a substitution $\xi:A\to A^{*}$ there is an associated transition matrix $M_{\xi}$, where $\{M_{\xi}\}_{ij}$ is the number of occurrences of $a_j$ in $\xi(a_i)$. A substitution $\xi$ is called \emph{irreducible} if for all $1\le i,j \le n$, there exists an exponent $k=k(i,j) \ge1$ such that the letter $a_i$ appears in the word $\xi^{k}(a_j)$. The substitution $\xi$ is called \emph{primitive} if $k$ can be chosen independently. in what follows, up to passing to powers and rearranging the letters, we will assume that each transition matrix is a lower diagonal block matrix where each diagonal block is either primitive, or has bounded entries for all $M^{t}$ for all $t\ge1$, \cite[Lemma 3.1]{LU1}. We refer reader to \cite{Q} and \cite{LU1} for a detailed account of substitutions. 

Given a non-primitive substitution we consider maximal invariant \emph{sub-alphabets} \[
0=A_0\sqsubset  A_1\sqsubset A_2\sqsubset\ldots\sqsubset A_n=A\]
and call $A_{i+1}\setminus A_{i}$ the $i^{th}$ stratum in analogy with train tracks terminology, \cite[Proof of Proposition 3.5]{LU1}.   

Given two words $w_1$ and $w_2$ in $A^{*}$, let $|w_1|_{w_2}$ denote the number of occurrences of the word $w_2$ in $w_1$.  The following is a slight variations of Theorem 1.2 and Corollary 1.3 of \cite{LU1}, a detailed proof of which is given in \cite[Proposition 5.4, Case 1]{LU1}. 

\begin{proposition}\cite{LU1}
\label{substitutionresult}
Let $\xi$ be a substitution 
on 
a finite alphabet $A$. 
Then there exists a positive power 
$\zeta=\xi^{s}$ 
such that
for 
any 
non-empty 
word 
$w \in A^*$ 
and any letter $a_i \in A$ 
the 
limit 
frequency
\[
\lim_{t\to\infty}\frac{|\zeta^{t}(a_i)|_w}{|\zeta^{t}(a_i)|}
\]
exists. Furthermore, if $a_i$ is in a primitive stratum $H_i$, where the Perron-Frobenius eigenvalue of $H_i$ is strictly bigger than those of the dependent strata, then the limit frequencies are independent of the chosen letter. \end{proposition}

The next Proposition shows how one can extract dynamical information from $\CT$'s by interpreting them as substitutions and invoking Proposition \ref{substitutionresult}.  Similar ideas were also used in our earlier work \cite{Uyaiwip, LU2} in the setting of train tracks and \cite{Gupta} in the $\CT$ setting for studying dynamics of \emph{relative} outer automorphisms.

For any two reduced edge paths $\gamma$ and $\gamma'$ in a graph $G$ define 
\[
\langle\gamma,\gamma'\rangle:=|\gamma'|_{\gamma}+|\gamma'|_{\gamma^{-1}}
\]

\begin{proposition}\label{freqconverge} Let $f:G\to G$ be a $\CT$ that represents an atoroidal outer automorphism $\varphi\in\Out(\FN)$. For any splitting unit $\sigma$, and any reduced edge path $\gamma$ in $G$ the limit 
\[
\sigma_{\gamma}:=\lim_{n\to\infty}\frac{\langle\gamma, f^t(\sigma)\rangle}{|f^t(\sigma)|}
\]
exists. Moreover, for any expanding splitting unit $\sigma$, the set of values $\{\sigma_{\gamma}\mid \textit{$\gamma$ is a reduced edge path in $G$}\}$ defines a geodesic current $\mu_\sigma$ on $\FN$.
\end{proposition}

\begin{proof}
If the splitting unit is not expanding then there is a definite bound on the length $|[f^{n}(\sigma)]|$ for all $n\ge1$. Hence the image $[f^{n}(\sigma)]$ becomes periodic after sufficiently many iterations. Since every periodic Nielsen path has period one, the sequence of paths $[f^{n}(\sigma)]$ becomes eventually fixed, and the claim follows.  For the remaining part of the proof we assume that $\sigma$ is an expanding splitting unit and will prove the claim by induction on the height of the stratum. Let $r=1$. Since $\varphi$ is atoroidal, $H_1$ is necessarily an $\EG$ stratum, and the restriction of $f$ to $G_1=H_1$ is an absolute train track map. Hence, the result follows from \cite[Proposition 2.4 and Lemma 3.7]{Uyaiwip}. Now assume that the claim holds for $r\le k-1$. There are three cases to consider. 

First suppose that $H_k$ is an $\EG$ stratum. A splitting unit of height $k$ is either an edge $e\in H_k$, or an $\INP$ intersecting $H_k$. Since an $\INP$ is not expanding we just need to prove the claim for an edge $e\in H_k$.
Let $A$ be the alphabet whose letters consists of edges in irreducible strata that are in $G_k$, $\INP$'s contained in $G_{k-1}$, and maximal, taken connecting paths in a zero stratum that are in $G_{k-1}$.  The fact that this alphabet is finite follows from the properties of the $\CT$ map that represents an atoroidal outer automorphism. Let  $\zeta:A^{*}\to A^{*}$ be the substitution induced by the $\CT$ $f:G\to G$ on the alphabet $A$ using the following rule: $\zeta(\sigma)=[f(\sigma)]$. For each ``letter'' in the above alphabet, the image is completely split and hence a reduced ``word'' in this alphabet. Hence the above formula is a substitution, and Proposition \ref{substitutionresult} gives the required convergence. 

The latter claim that the set of values $\{\sigma_{\gamma}\}_{\gamma\in\mathcal{P}G}$ defines a unique geodesic current is easy to check. They satisfy Kirchhoff conditions, i.e. 
\begin{enumerate}
\item $0\le \sigma_{\gamma}\le2<\infty$
\item $\sigma_{\gamma}=\sigma_{{\gamma}^{-1}}$
\item $\sigma_{\gamma}=\sum_{a\in A} \sigma_{a\gamma}=\sum_{a\in A} \sigma_{\gamma a}$ 
\end{enumerate}
as in \cite[Proposition 3.13]{LU1} and  \cite[Lemma 3.7]{Uyaiwip}, and by Kolmogorov measure extension theorem the result follows. 

Now assume that $H_k$ is an $\NEG$ stratum. Since $\sigma$ is expanding it is necessarily a superlinear edge $e$. By properties of $\CT$'s, $f(e)=e\cdot u$ where $u$ is a circuit in $G_{k-1}$ such that  $u$ is completely split and the turn $(u, u^{-1})$ is legal. We can similarly define a substitution as in $\EG$ case where the alphabet consists of the edge $e$, and splitting units appearing in $u$, and all of its iterates. The frequency convergence for the corresponding substitution is now given by Theorem \ref{substitutionresult}. 

Finally, if $H_k$ is a zero stratum, then $\sigma$ is a maximal connecting taken path, whose image $[f(\sigma)]$ is completely split, and has height $\le k-1$. Hence the claim follows by induction. 

\end{proof}

\begin{remark} Note that Proposition \ref{substitutionresult} together with the arguments in the proof of Proposition \ref{freqconverge} reveals that, for an $\EG$ stratum $H_r$ where the $\PF$-eigenvalue is strictly greater than those of the dependent strata, the currents $\mu_e$ are independent of the edge $e$ chosen from $H_r$. Furthermore, combined with \cite[Proposition 5.4]{LU1}, we have that for any other expanding splitting unit $\sigma$, the current $\mu_{\sigma}$ is a linear combination of currents coming from edges in $\EG$ strata. 
\end{remark}

\begin{definition} Given a $\CT$ map $f:G\to G$ that represents an atoroidal outer automorphism $\varphi\in\Out(\FN)$, we define the \emph{simplex of attraction} as the projective class of non-negative linear combinations of currents obtained from Proposition \ref{freqconverge}. We define the \emph{simplex of repulsion} similarly, using a $\CT$ map that represents $\varphi^{-1}$.  
\end{definition}

\subsection{Goodness and legal structure}

\begin{lemma}[Bounded Cancellation Lemma]\cite{Coo}\label{BCL} Let $f:G\to G$ be a topological graph map. There exists a constant $C_{f}$ such that for any reduced path $\rho=\rho_{1}\rho_{2}$ in $G$ one has
\[
|[f(\rho)]|\ge|[f(\rho_1)]|+|[f(\rho_2)]|-2C_{f}.
\]
That is, at most $C_f$ terminal edges of $[f(\rho_1)]$ are cancelled with $C_f$ initial edges of $[f(\rho_2)]$ when we concatenate them to obtain $[f(\rho)]$. 
\end{lemma}

\begin{definition}[Goodness]Let $\gamma$ be a reduced edge path in $G$ and $\gamma=\gamma_1\cdot\gamma_2\cdot\ldots\cdot\gamma_m$ be a splitting of $\gamma$ into edge paths $\gamma_i$. Define $\g_{CT}(\gamma)$ to be the proportion of the sum of the lengths of $\gamma_i$'s that have a complete splitting to the total length of $\gamma$. Define \emph{goodness} of $\gamma$, denoted $\g(\gamma)$, as the supremum of $\g_{CT}(\gamma)$ over all splittings of $\gamma$ into edge paths. Since there are only finitely many decompositions of an edge path into sub-edge paths the value $\g(\gamma)$ is realized for some splitting of $\gamma$. We will call the splitting for which $\g(\gamma)$ is realized the \emph{maximal edge splitting} of $\gamma$. The subpaths that are part of a complete splitting in the maximal edge splitting will be called \emph{good}. The subpaths in the maximal edge splitting which do not admit complete splittings will be called \emph{bad}. 

Let $w\in\FN$ be a conjugacy class in $\FN$, and $\gamma_w$ be the unique circuit in $G$ that represents $w\in\FN$. We define the \emph{goodness} of the conjugacy class $w$ as $\g(w):=\g(\gamma_w)$. 
\end{definition}

\begin{remark} The properties of $\CT$'s, see Lemma \ref{CTmaps} items (\ref{splits1}) and (\ref{splits2}), imply that forward images of good paths are always good, and forward images of bad paths are eventually good.  
\end{remark}

\begin{proposition} \label{lotsofexpanding} Let $f:G\to G$ be a $\CT$ representing an atoroidal outer automorphism $\varphi\in\Out(\FN)$. There exists $s>0$ such that for any completely split edge path $\sigma$ such that $|\sigma|$ is sufficiently big
\[
\frac{(\text{total length of expanding splitting units in } \sigma )}{|\sigma|}\ge s\]

\end{proposition}

\begin{proof} Let $\sigma$ be a completely split edge path, and consider its complete splitting. By properties of $\CT$'s (Lemma \ref{CTmaps} (\ref{zerolink})) each maximal connecting path in a zero stratum necessarily followed by an edge in an $\EG$ stratum. Since zero strata are precisely the contractible components, there is an upper bound for the length of any maximal connecting path in a zero stratum, say $Z_0$. Since $\varphi$ is atoroidal, there is also an upper bound for the length of any path that is a concatenation of $\INP$'s and fixed edges, say $Z_1$. Let $Z=\max\{Z_0,Z_1\}$. From these two observations it follows that for any completely split edge path of length$\ge 2Z+1$, we have 
\begin{align*}
& \frac{(\text{total length of expanding splitting units in } \sigma )}{|\sigma|} \ge \\ &
\frac{(\text{total length of } \EG \text{ or superlinear edges in } \sigma )}{|\sigma|}\ge
\frac{1}{2Z+1}
\end{align*}

\end{proof}

\begin{convention-remark}\label{passtoexpand} Note that the values $Z_0, Z_1$ and hence $Z$ are valid for all powers of $f$. From now on, we will replace $\varphi$, and hence $f$ with a power (which we will still denote by $f$) so that each expanding splitting unit grows at least by a factor of $2(2Z+1)$.  
\end{convention-remark}

\begin{definition}[Short and long good paths] In light of Proposition \ref{lotsofexpanding} we will call a good segment $\gamma$ \emph{long good segment} if $|\gamma|\ge2Z+1$ and \emph{short good segment} if $|\gamma|\le2Z$. 
\end{definition}

\begin{lemma}\label{boundonbad} Let $C_f$ be the bounded cancellation constant, and set $C:=\max\{C_f, 2Z+1\}$. Let $\gamma=\gamma_1\gamma_2$ be an edge path such that $\gamma_1$ and $\gamma_2$ are completely split. Then, any edge that is $C$ away from the turn $\{\gamma_1^{-1}, \gamma_2\}$ is good. 
\end{lemma}

\begin{proof} Since any completely split path of length $\ge 2Z+1$ grows at least by a factor of $2$, bounded cancellation lemma dictates that reducing $f(\gamma_1\gamma_2)$ will not result in any cancellation at edges $C$ away from the concatenation point, hence the claim follows. 
\end{proof}

\begin{lemma}\label{nomorebad} For any edge path $\gamma$ the total length of bad subpaths in $[f^{k}(\gamma)]$ is uniformly bounded by $|\gamma|2C$. 
\end{lemma}

\begin{proof} This is an easy consequence of Lemma \ref{boundonbad}. 
\end{proof}

We first show that, up to passing to further powers, the goodness is \emph{monotone}. 

\begin{lemma}\label{gettingbetter} Let $f:G\to G$ be a $\CT$ representing an atoroidal outer automorphism $\varphi\in\Out(\FN)$. There exists an exponent $t'\ge1$ such that for any circuit $\gamma$ with $1>\g(\gamma)>0$, and for all $t\ge t'$ one has
\[
\g([f^{t}(\gamma)])>\g(\gamma). 
\]
\end{lemma}

\begin{proof} Note, by definition the total length of good subpaths in $\gamma$ is $\g(\gamma)|\gamma|$. Under iteration of $f$, each good segment remains good and the length of each good segment is non-decreasing. Therefore, total length of good segments in $[f^{k}(\gamma)]$ is $\ge \g(\gamma)|\gamma|$. 

Let $t'$ be an exponent such that for each edge path $\beta$ of length $\le 2C+1$, the edge path $[f^{t}(\gamma)]$ is completely split for all $t\ge t'$. Therefore, for any bad segment $\beta$ such that $|\beta|\le 2C+1$, the path $[f^{t}(\beta)]$ is completely split, hence contains no bad edges.  For any bad segment $\beta$ of length $\ge 2C+1$, divide $\beta$ into subsegments $\beta_i$ of length $2C+1$, with the exception of last segment being length $\le 2C+1$. By the choice of $t'$, each $[f^{t}(\beta_i)]$ is completely split, where the turn at concatenation points are possibly illegal. Bounded cancellation lemma dictates  that total length of bad segments decreases by at least number of subsegments, and the conclusion of the lemma follows. 

\end{proof}

\begin{convention-remark}\label{monotoneexpand} In what follows, we pass to a further power of $\varphi$ and $f$ so that each expanding splitting unit grows at least by a factor of $2(2Z+1)$ and the goodness function is monotone. We furthermore consider the bounded cancellation constant for this new power, but we continue to use $f$ and $C_f$. 
\end{convention-remark}

The following is one of the key technical lemmas in this paper. It allows us to get convergence estimates while dealing with forward iterations of $\CT$'s. 

\begin{lemma}\label{improvinggoodness} Let $\delta>0, \epsilon>0$ be given. There exists an exponent $m_{+}=m_{+}(\delta,\epsilon)$ such that for all circuits $\gamma$ with $\g(\gamma)>\delta$, we have $\g([f^{m}(\gamma)])>1-\epsilon$ for $m\ge m_{+}$. 
\end{lemma}

\begin{proof}

Let $\gamma$ be a cyclically reduced edge path such that $\g(\gamma)=\delta>0$. First consider the splitting of $\gamma$ into maximal good segments $a_i$ and maximal bad segments $b_i$. There are two cases to consider: 

\noindent{\bf Case 1.} First assume that
\[
\frac{(\text{total length of long good segments in } \gamma )}{(\text{total length of good segments in }  \gamma)}\ge \frac{1}{4Z+1}
\]
This gives that 

\[
\frac{\text{total length of expanding splitting units in }\gamma}{\text{total length of good segments in }\gamma}\ge \frac{1}{(2Z+1)(4Z+1)}. 
\]

Note that by Lemma \ref{nomorebad} the total length of bad segments in $[f^{k}(\gamma)]$ is uniformly bounded by $(1-\g(\gamma))|\gamma|C$. On the other hand, the assumption above together with Convention-Remark \ref{monotoneexpand} implies that 
\[
\text{total length of good segments in $[f^{k}(\gamma)]$} \ge \g(\gamma)|\gamma|\frac{1}{(2Z+1)(4Z+1)}(2Z+1)^{k} 2^k. 
\]
Therefore,

\begin{align*}
\g([f^{k}(\gamma)] & \ge \dfrac{\g(\gamma)|\gamma|\frac{1}{(2Z+1)(4Z+1)}(2Z+1)^{k}2^k}{(1-\g(\gamma))|\gamma|C+\g(\gamma)|\gamma|\frac{1}{(2Z+1)(4Z+1)}(2Z+1)^{k}2^k} \\
& =\dfrac{\g(\gamma)\frac{1}{(4Z+1)}(2Z+1)^{k-1}2^k}{(1-\g(\gamma))C+\g(\gamma)\frac{1}{(4Z+1)}(2Z+1)^{k-1}2^k}
\end{align*}
which converges to $1$ as $k\to\infty$, hence the conclusion of Lemma \ref{gettingbetter} follows for big enough $k$, say $k=m_{+}$.  

\medskip

\noindent{\bf Case 2.} Otherwise we have,

\[
\frac{(\text{total length of long good segments in } \gamma )}{(\text{total length of good segments in }  \gamma)}< \frac{1}{4Z+1}. \]
Equivalently, 
\[
\frac{(\text{total length of short good segments in } \gamma )}{(\text{total length of long good segments in }  \gamma)}\ge 4Z. \tag{3.12.2}\label{qencase2}
\]

We now subdivide the path $\gamma$ into subpaths as follows.  Consider the maximal edge splitting of $\gamma$. First subpath starts at a good edge, and it stops after tracing a total length of $2Z+1$ good segments end at a vertex such that the next edge is good. The second subpath starts at where the first path stops, and traces a total length of $2Z+1$ good segments, and stops  at a vertex such that the next edge is good. We inductively form subpaths $\gamma_1,\gamma_2, \ldots$ so that each of them contains good segments of length $2Z+1$, with the possible exception of the last subpath. Note that by construction, $\gamma_i\cdot\gamma_2\cdot\ldots\cdot \gamma_s$ is a splitting of $\gamma$.

Observe that the equation \ref{qencase2} implies that 
\[
\frac{\#\{\gamma_i \text{ containing bad segments} \}}{\#\{\gamma_i\text{ which are completely good}\}}\ge 4Z
\]
which, in turn, implies,
\[
\#\{\gamma_i \text{ containing bad segments} \}\ge \frac{s4Z}{4Z+1},
\]
where $s$ is the total number of subpaths in $\gamma$ in above subdivision. 

Since 
\[
{\text{total length of good segments in }\gamma}\le\frac{1-\g(\gamma)}{\g(\gamma)}s(2Z+1),
\]
each $\gamma_i$ above that contains a bad segment, contains
\[
\frac{1-\g(\gamma)}{\g(\gamma)}s(2Z+1)\frac{4Z+1}{s(4Z)}=\frac{(2Z+1)(4Z+1)}{4Z}\frac{(1-\g(\gamma))}{\g(\gamma)}
\]
bad edges on \emph{average}. 

Therefore, for each $\gamma$ with $\g(\gamma)\ge\delta$, at least half of the subpaths contains bad segments of total length
\[
\le \frac{(2Z+1)(4Z+1)}{2Z}\frac{(1-\delta)}{\delta}=:C_b
\]

Let $t_b>0$ be an exponent such that for all edge paths $\gamma$ with $|\gamma|\le C_b$ the path $[f^{k}(\gamma)]$ is completely split for $k\ge t_b.$ Therefore, at least half of the subsegments in the subdivision will be mapped to long good segments, and the result follows from Case 1.

\end{proof}

\begin{lemma} \label{goforward} Let $U$ a neighborhood of the simplex of attraction and a positive number $\delta_{+}>0$ be given. Then, there exists an exponent $N=N(\delta,U)$ such that for any $w\in\FN$ with $\g(w)>\delta$,
\[
(\varphi^{N})^{n}(\eta_{w})\in U\]
for all $n\ge 1$. 
\end{lemma}

\begin{proof} We first apply a power of $f$ so that for every conjugacy class $w$ with $\g(w)>\delta$, we have $\g(\varphi(w))>1-\epsilon$ for small $\epsilon>0$. The rest of the proof is nearly identical to proof of Lemma 6.1 in \cite{LU2} where edges are replaced by expanding splitting units. \end{proof}

\begin{lemma}\label{goodbaddichotomy} Let $f:G\to G$ be a $\CT$ that represents an atoroidal outer automorphism. Given $0<\delta<1$, there exists an exponent $T$ such that, for any element $w\in\FN$, and for all $t\ge T$ either
\[
\g(\varphi^{t}(w))\ge\delta
\]
or 
\[
(\text{total length of bad segments in $f^{t}(\gamma_w)$})\le \frac{1}{2}(\text{total length of bad segments in $\gamma_w$}). 
\]
where $\gamma_w$ is the unique circuit in $G$ representing $w$. 

\end{lemma}

\begin{proof} Let $\gamma_w$ be the unique circuit in $G$ that represents $w\in\FN$. Consider the splitting of $\gamma$ into maximal good segments $a_i$ and maximal bad segments $b_i$. Recall that $C=\max\{C_f, 2Z+1\}$. Let us call a bad segment $b_i$ \emph{long bad segment} if $|b_i|>10C$, and \emph{short bad segment} otherwise. 

There are two cases to consider: 

\medskip
\noindent{\bf Case 1.} First assume that,
\[
\frac{\text{total length of short bad segments in $\gamma_w$}}{\text{total length of bad segments in $\gamma_w$}}\ge\frac{1}{10}. 
\]
Since every maximal bad segment is followed by at least one good segment, we have
\[
(\text{total length of good segments in $\gamma_w$})\ge\frac{1}{10C} (\text{total length of short bad segments in $\gamma_w$})
\]
and hence 
\[
(\text{total length of good segments in $\gamma_w$})\ge\frac{1}{100C} (\text{total length of bad segments in $\gamma_w$}).
\]
Therefore, 
\[
\g(\gamma_w)\ge\frac{1}{100C+1}.
\]
Now, invoking Lemma \ref{improvinggoodness}, there is an exponent $T_1$ such that 
\[
\g(\varphi^{t}(w))\ge\delta
\]
for all $t\ge T_1$, which is clearly independent of the conjugacy class $w$.  

\medskip
\noindent{\bf Case 2.} 
Now assume otherwise that,
\[
\frac{\text{total length of long bad segments in $\gamma_w$}}{\text{total length of bad segments in $\gamma_w$}}\ge\frac{9}{10}. \tag{3.17.2}\label{muchlongbad}
\]

Let $T_2$ be an exponent such that for all edge paths $\gamma$ with $|\gamma|<10C$, $[f^{t}(\gamma)]$ is completely split for all $t\ge T_2$. Then for any long bad segment $b$, bounded cancellation lemma implies that
\[
\text{total length of bad segments in $[f^{t}(b)]$}\le\frac{1}{5}\text{total length of bad segments in $b$}. 
\]
Together with the assumption \ref{muchlongbad}, we get 

\[
(\text{total length of bad segments in $f^{t}(\gamma_w)$})\le \frac{9}{50}(\text{total length of bad segments in $\gamma_w$}). 
\]
for all $t\ge T_2$. Now set $T=\max\{T_1,T_2\}$, and the lemma follows. 

\end{proof}

\begin{lemma}\label{futurepast} Let $h:G'\to G'$ be a $\CT$ that represents $\varphi^{-1}\in\Out(\FN)$. Define, $\g'(\gamma')$ for $\gamma'\in G'$, and $\g'(w)$ for $w\in\FN$ analogously. Then, given $0<\delta<1$, there is an exponent $T>0$ such that, up to replacing $f$ and $h$ with powers, for any element $w\in\FN$, either
\[
\g(\varphi^{t}(w))\ge\delta
\]
or 
\[
\g'(\varphi^{-t}(w))\ge\delta
\]
for all $t\ge T$. 
\end{lemma}

\begin{proof}  Let $h:G'\to G'$ be a $\CT$ that represents $\varphi^{-1}\in\Out(\FN)$ and $\g'$ be the corresponding goodness function, and we pass to appropriate powers according to Convention-Remark \ref{monotoneexpand}. The proof is now nearly identical to that of \cite[Proposition 4.20]{LU2} where number of illegal turns is replaced by total length of bad segments. 
\end{proof}

\begin{proposition}\cite[Proposition 3.3]{LU2}
\label{convergence-criterion} Let $f:X\to X$ be a homeomorphism of a compact 
metrizable 
space $X$. 
Let 
$Y \subset X$ be dense subset of $X$, and let $\Delta_{+}$ and $\Delta_{-}$ be two $f$-invariant sets in $X$ that are disjoint. Assume that the following criterion holds:

For every neighborhood $U$ of $\Delta_+$ and every neighborhood $V$ of $\Delta_-$ there exists an integer $m_0 \geq 1$ such that for any $m \geq m_0$ and any $y \in Y$ one has either $f^m(y) \in U$ or $f^{-m}(y) \in V$. 

Then $f^{2}$ has generalized uniform North-South dynamics from $\Delta_{-}$ to $\Delta_+$.
\end{proposition}

\begin{proposition}\cite[Proposition 3.4]{LU2} 
\label{NS-for-roots}
Let $f: X \to X$ be a homeomorphism of a compact 
space $X$, and let $\Delta_{+}$ and $\Delta_{-}$
be disjoint $f$-invariant sets. Assume that some power $f^p$ with $p \geq 1$ has generalized uniform North-South dynamics from $\Delta_-$ to $\Delta_+$. 

Then the map $f$, too, has generalized uniform North-South dynamics from $\Delta_-$ to $\Delta_+$.
\end{proposition}

\begin{proof}[Proof of Theorem \ref{dynamicsofhyp}] The theorem now follows from combination of Lemma \ref{goforward}, Lemma \ref{futurepast}, Proposition \ref{convergence-criterion} and Proposition \ref{NS-for-roots}.  
\end{proof}

\section{Hyperbolic extensions of free groups}

In this section we use the dynamics of atoroidal outer automorphisms to prove Theorem \ref{maintwo} from the introduction which allows us to construct new examples of hyperbolic extensions of free groups. 

In what follows we will utilize theory of laminations on free groups which appear as supports of currents on $\FN$. We refer reader to \cite{BFH97, BFH00, CPJTA, CHL1, CHL2, CHL3, FH, HMIntro} for detailed discussions. A \emph{lamination} is a closed subset of $\partial^{2}\FN$ which is $\FN$-invariant, and flip-invariant. We say that a free factor $\F$ \emph{carries} a lamination $\Lambda$ if all lines in $\Lambda$ are contained in $\partial^{2}\F$.  

\begin{convention} Throughout this section we assume that we pass to the finite index characteristic subgroup $\IA_{N}(\ZZ_3)$ of $\Out(\FN)$ as in Handel--Mosher subgroup decomposition theory \cite{HMIntro} so that for each outer automorphism every periodic conjugacy class is fixed, and every periodic free factor system is invariant. 
\end{convention}

Let $\calH$ be a subgroup of $\Out(\FN)$ and $\calF^1\sqsubset\calF^2\sqsubset\ldots\sqsubset\calF^n=\FN$ be a \emph{maximal} $\calH$-invariant filtration of $\FN$ by free factor systems, meaning that if $\calH(\calA)=\calA$ for some $\calF^{i}\sqsubset\calA\sqsubset\calF^{i+1}$, then either $\calA=\calF^{i}$ or $\calA=\calF^{i+1}$. Let $\varphi\in\calH$ be an atoroidal outer automorphism. Consider a (possibly trivial) refinement $\calA_1\sqsubset\calA_2\sqsubset\ldots\sqsubset\calA_m=\FN$ of $\calF^1\sqsubset\calF^2\sqsubset\ldots\sqsubset\calF^n=\FN$ which is a maximal invariant filtration for $\varphi$. 

If $\calH$ fixes the conjugacy class of a free factor $\F$ of $\FN$, we will call the image of $\calH$ in $\Out(\F)$ under the natural homomorphism $\Stab(\F)\to\Out(\F)$ \emph{the restriction} of $\calH$ to $\F$ and denote it by $\calH_{\mid\F}$. 

We say that an $\calH$-invariant free factor $\F$ is \emph{minimal} if $\calH$ does not fix the conjugacy class of any proper free factor of $\F$. Similar definition holds for $\varphi$ by considering the cyclic subgroup $\langle \varphi\rangle$. Observe that for $\varphi\in\calH$ each minimal $\varphi$-invariant free factor $\F^{i}_{\varphi}$ is contained in a unique minimal $\calH$-invariant free factor $\F^{i}_{\calH}$.

\begin{definition}\label{independent} Let $\varphi$ and $\psi$ be two atoroidal outer automorphisms with attracting and repelling simplices $\Delta_{\pm}(\varphi)$  and $\Delta_{\pm}(\psi)$ given by Theorem \ref{dynamicsofhyp}.  We say that $\varphi$ and $\psi$ are independent if $\Delta_{\pm}(\varphi)\cap\Delta_{\pm}(\psi)=\emptyset$. 
\end{definition}

\begin{lemma}\label{goodconjugator} Let $\varphi\in\mathcal{H}$ be an atoroidal outer automorphism. 
Suppose that the restriction of  $\calH$ to $\F_i$
is not virtually cyclic for each minimal $\calH$-invariant free factor $\F_i$ of $\FN$. Then $\calH$ contains two independent atoroidal outer automorphisms.

\end{lemma}

\begin{proof}

Let $\{\F^i\}_{i=1}^{s}$ be the set of all minimal $\calH$-invariant free factors. For each $i=1, \ldots s$ the restriction of $\calH$ to $\F^i$ is irreducible, hence by \cite[Theorem A]{HMIntro} and \cite[Theorem 0.1]{Hshort} $\calH$ contains an element $\theta_i$ whose restriction to $\F^i$ is fully irreducible, and since $\calH$ is not geometric (since it contains an atoroidal element), we can choose $\theta_i$ in a way that its restriction to $\F^i$ is both fully irreducible and atoroidal \cite[Theorem 5.4]{UyaNSD}. 

Since fully irreducible and atoroidal elements are precisely the loxodromic isometries of the \emph{co-surface graph} \cite{DTcosurface}, invoking \cite[Theorem 5.1]{CU} we can find a single $\theta\in\calH$ such that for each $i=1,\ldots s$, the restriction $\theta_{\mid\F^i}$ is fully irreducible and atoroidal. 
Recall that each fully irreducible and atoroidal outer automorphism acts on the space of projectivized geodesic currents with uniform north-south dynamics \cite{Martin, Uyaiwip}. Let $[\mu_{-}^{i}(\theta)]\in\PCurr( \F^{i}) $ and $[\mu_{+}^{i}(\theta)]\in\PCurr(\F^{i})$ denote the unstable and stable currents for the restriction $\theta_{\mid\F^i}$. Since the stabilizer of the set $\{[\mu_{-}^{i}(\theta)],[\mu_{+}^{i}(\theta)]\}$ is virtually cyclic in $\Out(\F^i)$, \cite{KL5}, using the assumption on $\calH$ we can furthermore assume  that for the above $\theta\in\calH$ it holds that $\theta_{\mid \F^i}$ and $\varphi_{\mid \F^i}$ are independent. 

Hence, we can find $M>0$ large enough so that $\theta^{M} (\Delta^{i}_{\pm}(\varphi))\cap\Delta^{i}_{\pm}(\varphi)=\emptyset$, where $\Delta^{i}_{\pm}(\varphi)$ are the attracting and repelling simplices of $\varphi_{\mid \F^{i}}$ in $\PCurr( \F^{i})$. 

More precisely, choose $U_i,V_i$ open neighborhoods of $[\mu_{+}^i(\theta)],[\mu_{-}^i(\theta)]$ in $\PCurr(\F^{i})$ which are disjoint from $\Delta^{i}_{\pm}(\varphi)$. Pick $M>0$ such that $\theta^{m}(\PCurr(\F^{i})\setminus V_i)\subset U_i$ for all $m\ge M$, in particular $\theta^{M}(\Delta^{i}_{\pm})\subset U_i$. See Figure \ref{NSDinminimal}. 
In fact, we choose $M$ that works for all minimal $\calH$-invariant free factors for suitable open neighborhoods of attracting simplices as there are only finitely many minimal $\calH$-invariant free factors.

\begin{figure}[h!]
\labellist
\tiny\hair 0.5pt
 \pinlabel {$V_i$} [ ] at 420 610
  \pinlabel {$\mu_{-}^{i}(\theta)$} [ ] at 469 595
 \pinlabel {$\bullet$} [ ] at 470 610
 \pinlabel {$U_i$} [ ] at 220 610
  \pinlabel {$\bullet$} [ ] at 180 590
  \pinlabel {$\mu_{+}^{i}(\theta)$} [ ] at 170 580
\pinlabel {$\Delta_{+}^i(\varphi)$} [ ] at 360 680
 \pinlabel {$\Delta_{-}^i(\varphi)$} [ ] at 360 540
 \pinlabel {$\PCurr(\F^{i})$} [ ] at 150 710
\endlabellist
\centering
\includegraphics[scale=0.65]{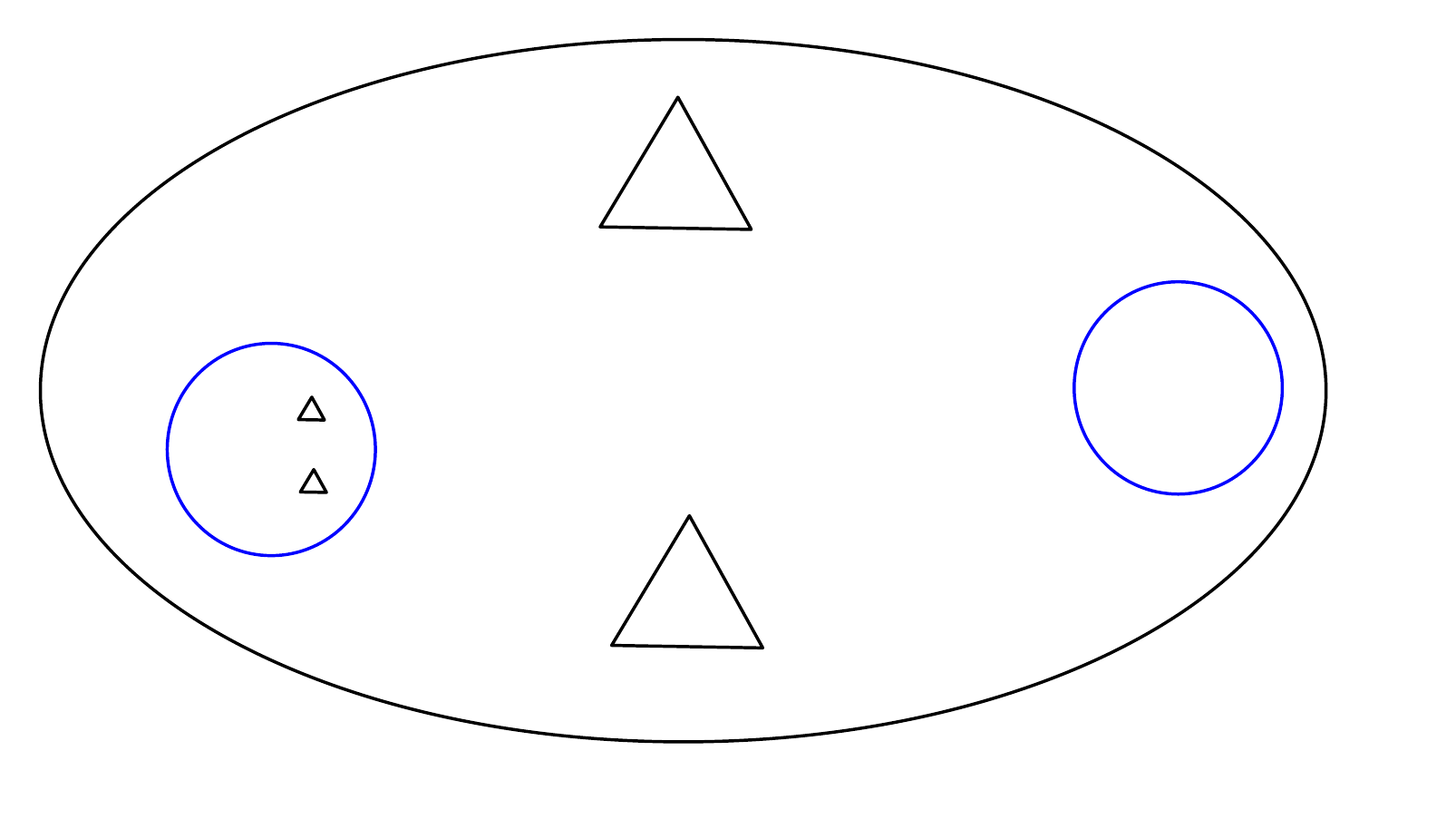}
\caption{Dynamics on $\PCurr(\F^{i})$}
\label{NSDinminimal}
\end{figure}

Now consider the automorphism $\eta=\theta^{M}\varphi \theta^{-M}$ which is atoroidal since being atoroidal is invariant under conjugacy. Furthermore $\Delta^{i}_{+}(\eta)=\theta^{M}(\Delta^{i}_{+}(\varphi))$ and $\Delta_{-}^{i}(\eta)=\theta^{M}(\Delta^{i}_{-}(\varphi))$ in $\PCurr(\F^{i})$,  and $\Delta_{+}(\eta)=\theta^{M}(\Delta_{+}(\varphi))$ and $\Delta_{-}(\eta)=\theta^{M}(\Delta_{-}(\varphi))$ in $\PCurr(\FN)$.

\begin{claim*} $\eta$ and $\varphi$ are independent. 
\end{claim*}

%

Let $[\mu_{+}^{k}]$ be an extremal point in the attracting simplex $\Delta_{+}(\varphi)$. We first want to show that $\theta^{M}([\mu_{+}^{k}])\neq [\mu]$ for any point $[\mu]\in\Delta_{\pm}(\varphi)$.

Notice that by Proposition \ref{freqconverge} the point $[\mu_{+}^{k}]$ corresponds to some $\EG$ stratum $H_k$ in the $\CT$ map $f:G\to G$ that represents $\varphi$. 

Let $\F^k_{\varphi}$ be the unique (minimal) free factor carrying $\supp(\mu_{+}^k)$ (this support is the \emph{attracting lamination} corresponding to the $\EG$ stratum $H_k$ in the sense of Bestvina--Feighn--Handel \cite{BFH00}), and consider a minimal $\varphi$-invariant free factor $\F^i_{\varphi}\sqsubset\F^k_{\varphi}$. The free factor $F^{i}_{\varphi}$ is contained in some minimal $\calH$-invariant free factor $F^{i}$ as above. 

Let $[\mu_{+}^{i}]\in\Delta_{+}$  be the unique geodesic current whose support is carried by $\F^i_{\varphi}$. We first observe that by definition $\supp(\mu_{+}^{i})\subset\supp(\mu_{+}^k)$. Second, the subgroup $\calH$ and so the element $\theta^{M}$ preserves $\F^{i}$, therefore $\supp(\theta^{M}\mu_{+}^{i})$ is carried by $\F^{i}$. 

Suppose, for the sake of contradiction, that $\theta^{M}([\mu_{+}^{k}])= [\mu]$. In that case we have $\supp(\theta^{M}([\mu_{+}^{k}]))=\supp([\mu])$, hence 
\[
\supp(\theta^{M}([\mu_{+}^{i}]))\subset\supp(\theta^{M}([\mu_{+}^{k}]))=\supp([\mu]). 
\]

Only sublaminations of $\supp([\mu])$ that are carried by $\F^{i}$ could possibly come from supports of extremal points of $\Delta_{\pm}^{i}(\varphi)$, and since $(\theta^{M}\Delta_{+}^{i})\cap\Delta_{\pm}^{i}=\emptyset$, the above inclusion is not possible, hence we get a contradiction. 

Since the support of any point in $\Delta_{+}(\eta)$ is union of supports of extremal points, we get  $\Delta_{+}(\eta)\cap\Delta_{\pm}(\varphi)=\emptyset$. A symmetric argument finishes the proof.  

\end{proof}

We will prove the hyperbolicity of the extension using a classical argument of Bestvina, Feighn and Handel \cite{BFH97} which originates in the work of Mosher \cite{Mosher} as interpreted by Kapovich and Lustig \cite{KL5}. 

\begin{proposition}\label{flare} Let $\varphi,\psi\in\mathcal{H}$ be two independent atoroidal outer automorphisms. Then, there exist $M,N>0$ such that for any $\mu\in\Curr(\FN)$ and for all $n\ge N, m\ge M$, for at least three out of four elements $\alpha$ in $\{\varphi^{n}, \varphi^{-n}, \psi^{m}, \psi^{-m}\}$  
\[
|\alpha\mu|_{G}\ge 2|\mu|_{G}. 
\]
\end{proposition}

\begin{proof} Let $U$ be a sufficiently small open neighborhood of $\Delta_{+}(\varphi)$, and $M_0>0$ so that for any $\mu\in\Curr(\FN)$ such that $[\mu]\in\Delta_{+}(\varphi)$ it holds that $|\varphi^{n}(\mu)|_{G}\ge 2 |\mu|_{G}$ for all $n\ge M_0$. This can be done, because of the topology of the space of currents, and the fact that for each extremal point $[\mu_{+}]$ of $\Delta_{+}(\varphi)$, $\varphi(\mu_{+})=\lambda \mu_{+}$ for some $\lambda>1$. For the corresponding statement in the fully irreducible case see \cite[Lemma 4.12]{KL5}. 

We also choose a small neighborhood $V$ of $\Delta_{-}(\varphi)$, and $M_1$ so that for each $\mu\in\Curr(\FN)$ such that $[\mu]\in\Delta_{-}(\varphi)$ it holds that  $|\varphi^{-n}(\mu)|_{G}\ge 2 |\mu|_{G}$ for all $n\ge M$. Let $M'=\max\{M_0, M_1\}$. 

Similarly, we choose neighborhoods $U'$ and $V'$ of $\Delta_{+}(\psi)$ and $\Delta_{-}(\psi)$ respectively and a corresponding $N'>0$. 

By Theorem \ref{dynamicsofhyp}, there exists an exponent $M_{+}$ such that \[
\varphi^{n}(\PCurr(\FN)\setminus V)\subset U\] and \[\varphi^{-n}(\PCurr(\FN)\setminus U)\subset V\] for all $n\ge M_{+}$. 

Similarly, there exists an exponent $N_{+}$ such that \[
\psi^{n}(\PCurr(\FN)\setminus V')\subset U'\]
and \[
\psi^{-n}(\PCurr(\FN)\setminus U')\subset V'\]
for all $n\ge N_{+}$.

Now set $M=M'+M_{+}$ and $N=N'+N_{+}$. Let $\mu\in V$. Then, the choice of $M$ and $N$ guarantee that $|\varphi^{-n}\mu|_{G}\ge 2|\mu|_{G}$, $|\psi^{m}\mu|_{G}\ge 2|\mu|_{G}$ and $|\psi^{-m}\mu|_{G}\ge 2|\mu|_{G}$. Other cases can be proved similarly, hence the proposition follows. 

\end{proof}

\begin{proof}[Proof of Theorem \ref{maintwo}]

Let ${F}^i$ be a minimal, non-trivial $\calH$-invariant free factor, and let $\varphi, \eta, \theta$ be as in Lemma \ref{goodconjugator}. Since $\theta$ is fully irreducible, for large $M$ the free factors $F^{i}_{\varphi}$ and $\theta^{M}F^{i}_{\varphi}$ \emph{fill} the free factor $\F^{i}$. Under this assumption, based on work of Bestvina and Feighn \cite{BFsubfactor}, Taylor \cite[Theorem 1.3]{ar:Taylor14} proved that for some $K>0$, the group $\langle \varphi^{K},\eta^{K}\rangle$ is isomorphic to a free group of rank $2$. (He proves much more but we don't need them here.)

The fact that the corresponding free group extension is hyperbolic now follows from Proposition \ref{flare} and Bestvina--Feighn combination theorem, see proof of \cite[Theorem 5.2]{BFH97}.  

\end{proof}

We finish the paper with Corollary \ref{irreduciblenotcoco} from the introduction: 

\begin{restate}{Corollary}{irreduciblenotcoco} Let $\varphi\in\Out(\FN)$ be a fully irreducible and atoroidal outer automorphism. Then, for any atoroidal outer automorphism $\psi\in\Out(\FN)$ (not necessarily fully irreducible) which is not commensurable with $\varphi$, there exists an exponent $M>0$ so that for all $n,m>M$,  the subgroup $\Gamma=\langle\varphi^{n},\psi^m\rangle<\Out(\FN)$ is purely atoroidal and the corresponding free extension $E_{\Gamma}$ is hyperbolic. 
\end{restate}

\begin{proof} Let $\varphi$ be as above, and let $[\mu_{+}(\varphi)]$, and $[\mu_{-}(\varphi)]$ be the corresponding stable and unstable currents in $\PCurr(\FN)$. Since $\psi$ is not commensurable with $\varphi$, the attracting simplex $\Delta_{+}(\psi)$, the repelling simplex $\Delta_{-}(\psi)$, and the stable and unstable currents $[\mu_{+}(\varphi)]$ and $[\mu_{-}(\varphi)]$ are all disjoint. 

Choose disjoint open neighborhoods of these sets, and choose high enough powers of $\varphi$ and $\psi$ so that there is a uniform north-south dynamics which is guaranteed by Theorem \ref{dynamicsofhyp}. Then Proposition \ref{flare}, together with Bestvina--Feighn combination theorem gives the required result. 
 
\end{proof}


\bibliography{bib}
\bibliographystyle{acm}

\end{document}